\documentclass[12pt]{amsart}
\usepackage{amssymb}
\usepackage[?]{amsrefs}

\newtheorem{Thm}{Theorem}[section]

\newtheorem{Cor}[Thm]{Corollary}
\newtheorem{Lem}[Thm]{Lemma}
\newtheorem{Prop}[Thm]{Proposition}

\theoremstyle{definition}

\theoremstyle{remark}

\def\ldots{\mathinner{\ldotp\ldotp\ldotp}}
\def\cdots{\mathinner{\cdotp\cdotp\cdotp}}

\def\tfrac{\textstyle\frac}

\def\C{\mathcal C}
\def\tfrac#1#2{{\textstyle\frac#1#2}}

\def\rank{\text{rank }}

\begin{document}

\title{Asymptotic unconditionality}

\author{S. R. Cowell}

\author{N. J. Kalton}

\address{Department of Mathematics \\
University of Missouri-Columbia \\
Columbia, MO 65211 }

\email{simon@math.missouri.edu}
\email{nigel@math.missouri.edu}


\thanks{The authors were supported by NSF grant DMS-0555670}

\begin{abstract}  We show that a separable real Banach space embeds almost isometrically in a space $Y$ with a shrinking 1-unconditional basis if and only if $\lim_{n\to\infty}\|x^*+x_n^*\|=\lim_{n\to\infty}\|x^*-x_n^*\|$ whenever $x^*\in X^*$, $(x_n^*)_{n=1}^{\infty}$ is a weak$^*$-null sequence and both limits exist. If $X$ is reflexive then $Y$ can be assumed reflexive. These results provide the isometric counterparts of recent work of Johnson and Zheng. \end{abstract}

\maketitle

\section{Introduction}\setcounter{equation}{0}

In this paper we consider only real Banach spaces.
Recently, Johnson and Zheng \cite{JohnsonZheng2008} gave an intrinsic characterization of  separable  Banach spaces which embed isomorphically into a reflexive  Banach space with unconditional basis.  Precisely a separable reflexive Banach space $X$ embeds into a (reflexive) Banach space with unconditional basis if and only if $X$ has the unconditional tree property (UTP), i.e. for some $C$, every weakly null tree has a $C$-unconditional branch.
The use of tree properties to describe subspaces of certain Banach spaces is a recent development in Banach space theory which originates in \cite{Kalton2001} and was later developed in \cite{OdellSchlumprecht2002}.

The results of \cite{Kalton2001} and \cite{OdellSchlumprecht2002} are both, in a certain sense, isomorphic versions of earlier isometric results from \cite{KaltonWerner1995}.  In the latter paper, for $1<p<\infty$, it is shown that if $X$ is  a separable Banach space containing no copy of $\ell_1$, then  $X$ $(1+\delta)-$embeds in an $\ell_p-$sum of finite-dimensional spaces for every $\delta>0$ if and only if
$$\lim_{n\to\infty}(\|x+x_n\|^p-\|x\|^p-\|x_n\|^p)=0$$ whenever $x\in X$ and $(x_n)_{n=1}^{\infty}$ is a weakly null sequence.  Similarly, again assuming $X$ is separable and contains no copy of $\ell_1$, $X$ $(1+\delta)-$embeds into $c_0$ for every $\delta>0$ if and only if
$$ \lim_{n\to\infty}(\|x+x_n\|-\max(\|x\|,\|x_n\|))=0$$ whenever $x\in X$ and $(x_n)_{n=1}^{\infty}$ is weakly null.

In \cite{Kalton2001} it was shown that a separable Banach space $X$, containing no copy of $\ell_1,$ embeds isomorphically into $c_0$ if and only if every weakly null tree has a $c_0$-branch; the corresponding  result for $1<p<\infty$ was given in \cite{OdellSchlumprecht2002} where it was shown that a reflexive Banach space $X$ embeds isomorphically into an $\ell_p-$sum of finite-dimensional spaces if and only if every weakly null tree has an $\ell_p$-branch.  We remark that in \cite{Kalton2001} the proof of the isomorphic result was given by renorming and reducing to a situation very similar to the isometric result.

The aim of this paper is to prove an isometric analogue of the Johnson-Zheng theorem. We say that a separable Banach space $X$ has {\it property $(au)$} if given any $x\in X$ and $\delta>0$ there is a closed subspace $F$ of finite codimension such that
$$ \|x-y\|\le (1+\delta)\|x+y\|, \qquad y\in F.$$  This could be restated as
$$ \lim_{d\in D}(\|x+x_d\|-\|x-x_d\|)=0$$ whenever $x\in X$ and $(x_d)_{d\in D}$ is a bounded weakly null net.  If $X$ has separable dual we may replace nets by sequences in this definition.
There is also a natural dual notion; a separable Banach space $X$ has {\it property $(au^*)$} if given any $x^*\in X^*$ and $\delta>0$ there is a weak$^*$ closed subspace $F$ of finite codimension in $X^*$ such that
$$ \|x^*-y^*\|\le (1+\delta)\|x^*+y^*\|, \qquad y^*\in F.$$  This is equivalent to
 $$ \lim_{n\to\infty}(\|x^*+x^*_n\|-\|x^*-x^*_n\|)=0$$ whenever $x^*\in X^*$ and $(x^*_n)_{n=1}^{\infty}$ is a weak$^*$ null sequence in $X^*$.  Both these concepts already exist in the literature under different names (see \cite{Sims1994} and \cite{Lima1995}).
It is easy to show that $(au^*)$ implies $(au)$ (Proposition \ref{*-auau} below) but the converse is false (take $X=\ell_1$).

Our main result (Theorem \ref{main}) is that a separable Banach space $X$ has property $(au^*)$ if and only if for every $\delta>0$ there is a Banach space $Y$ with a shrinking 1-unconditional basis and a subspace $X_{\delta}$ of $Y$ with $d(X,X_{\delta})<1+\delta;$ $Y$ may be assumed reflexive when $X$ is reflexive.  A special case of this theorem was already implicit in the literature.  Recall that a separable Banach space $X$ has the {\it unconditional metric approximation property (UMAP)} \cite{CasazzaKalton1990} if there is a sequence of finite rank operators such that $\lim_{n\to\infty}T_nx=x$ for $x\in X$ and $\lim_{n\to\infty}\|I-2T_n\|=1;$ if additionally $\lim_{n\to\infty}T_n^*x^*=x^*$ for $x^*\in X^*$ we say that $X$ has shrinking (UMAP).  Lima \cite{Lima1995} showed that if $X$ is a separable Banach space with property $(au^*)$ and such that $X^*$ has the approximation property then $X$ has (shrinking) (UMAP).
In \cite{GodefroyKalton1997} (Corollary IV.4) it is shown that if $X$ has shrinking (UMAP) then $X$ can be $(1+\delta)-$embedded in a space with a shrinking 1-unconditional basis; unfortunately the proof of this result is inaccurate (as Haskell Rosenthal has pointed out to us) and we give a corrected proof below (contained in Proposition \ref{embedding}).  Thus the novelty in Theorem \ref{main} is the removal of the approximation property hypothesis.
Let us also remark at this point that Johnson and Zheng \cite{JohnsonZhenginprep} have informed us that they have extended the methods of \cite{JohnsonZheng2008} to show that a separable Banach space $X$ with separable dual embeds isomorphically into a space with a shrinking unconditional basis if and only if $X^*$ has the weak$^*$-(UTP).  This provides a complete isomorphic analogue of Theorem \ref{main}.

If $X$ is reflexive $(au)$ is equivalent to $(au^*)$ and so Theorem \ref{main} could be restated using property $(au).$
We conjecture that if $X$ contains no copy of $\ell_1$ then $(au)$ and $(au^*)$ are equivalent.  We are not quite able to prove this, but we do prove a result very close to it.  We say that a separable Banach space has {\it property (WABS) (weak alternating Banach-Saks property)} if given any bounded sequence $(x_n)_{n=1}^{\infty}$ we can find a sequence of convex blocks $(y_n)_{n=1}^{\infty}$ such that
$$ \lim_{n\to\infty}\sup_{r_1<r_2<\cdots<r_n}\|\frac1n\sum_{j=1}^n(-1)^jy_{r_j}\|=0.$$
This condition is implied by reflexivity or the Alternating Banach-Saks property.  Then $X$ has property $(au^*)$ if and only if $X$ has property $(au)$ and (WABS).  The example of the James space \cite{James1951a} shows then there is a space with separable dual and (UTP) which has no equivalent renorming to have property $(au)$.

{\bf Acknowledgements.}  We would like to thank Gilles Godefroy, Vegard Lima  and Lova Randrianarivony for some helpful comments.

\vskip.5truein

\section{Asymptotic unconditionality}\setcounter{equation}{0}

Let $X$ be a separable Banach space.  we will say that $X$ is {\it asymptotically unconditional (au)} if given any $x\in X$ and $\delta>0$ there is a closed finite co-dimensional subspace $W$ of $X$ such that
$$ \|x-w\|\le (1+\delta)\|x+w\|, \qquad w\in W.$$
An alternative formulation of this condition is that
$$ \lim_{d\in D}\left(\|x+u_d\|-\|x-u_d\|\right)=0$$ whenever $x\in X$ and $(u_d)_{d\in D}$ is a
 bounded weakly null net.

We shall say that $X$ is {\it sequentially asymptotically unconditional ($\omega$-au)} if $$ \lim_{n\to\infty}(\|x+u_n\|-\|x-u_n\|)=0$$ whenever $x\in X$ and $(u_n)_{n=1}^{\infty}$ is weakly null sequence.  This condition has already been considered by Sims \cite{Sims1994} under the acronym WORTH.  Note that if $X^*$ is separable then the weak topology is metrizable on bounded sets and so $X$ is ($\omega$-au) if and only if $X$ is (au).

We shall say that $X$ is {\it *-asymptotically unconditional (au$^*$)} if
$$ \lim_{n\to\infty}(\|x^*+x^*_n\|-\|x^*-x^*_n\|)=0$$ whenever $x^*\in X^*$ and $(x^*_n)_{n=1}^{\infty}$ is a weak$^*$-null sequence in $X^*$.  This condition has been considered under the name $(wM^*)$ by Lima \cite{Lima1995}; later Oja \cite{Oja2000} considered a family of more general conditions.
Since $X$ is assumed separable, the weak$^*$-topology on bounded sets is metrizable, and so $X^*$ is *-asymptotically unconditional if and only if either
given any $x^*\in X^*$ and $\epsilon>0$ there is a weak$^*$-closed finite co-dimensional subspace $W$ of $X^*$ such that
$$ \|x^*-w^*\|\le (1+\epsilon)\|x^*+w^*\|, \qquad w^*\in W,$$
or,  alternatively,
$$ \lim_{d\in D}(\|x^*+u^*_d\|-\|x^*-u^*_d\|)=0$$ whenever $x^*\in X$ and $(u^*_d)_{d\in D}$ is a bounded weak$^*$-null net.

We first state a very simple principle based on compactness that will be used frequently:

\begin{Lem}\label{principle} (i) Let $X$ be a separable Banach space with property $(au)$.  Then given any finite-dimensional subspace $E$ of $X$ and $\delta>0$ there is a closed finite codimensional subspace $F$ of $X$ such that $$\|e-f\|\le (1+\delta)\|e+f\|, \qquad e\in E, f\in F.$$
(ii) Let $X$ be a separable Banach space with property $(au^*)$.  Then given any finite-dimensional subspace $E$ of $X^*$ and $\delta>0$ there is a weak$^*$-closed finite codimensional subspace $F$ of $X^*$ such that $$\|e^*-f^*\|\le (1+\delta)\|e^*+f^*\|, \qquad e^*\in E, f^*\in F.$$\end{Lem}

The following result is a consequence of \cite{Lima1995} Proposition 4.1, but we give an independent proof.

\begin{Prop}\label{separabledual} If $X$ is a separable  Banach space with (au*) then $X^*$ has no proper norming subspace and hence is separable.\end{Prop}

\begin{proof} Let $M$ be a norming subspace of $X^*$.  If $M\neq X^*$ then there exists $x^*\in X^*$ with $\|x^*\|=1$ and $d(x^*,M)=d>1/2.$  Let $(x_n^*)$ be a sequence in $B_{X^*}\cap M$ which weak$^*$-converges to $x^*.$  Then
$$ 1\ge \lim_{n\to\infty}\|x_n^*\|=\lim_{n\to\infty}\|x^*+(x_n^*-x^*)\|=\lim_{n\to\infty}\|2x^*-x_n^*\|\ge 2d>1.$$  This contradiction gives the result.\end{proof}

\begin{Prop}\label{*-auau} Suppose $X$ is a separable Banach space.  Then\newline
(a) If $X$ has (au*) then $X$ has (au).\newline
(b) If $X$ is reflexive then $X$ has (au*) if and only if $X$ has (au).
\end{Prop}

\begin{proof}  (a): It is enough to show that if $x\in X$ and $(u_d)_{d\in D}$ is a bounded weakly null net with
$$ \lim_{d\in D}\|x+u_d\|=1,\qquad \lim_{d\in D}\|x-u_d\|=\theta$$ then $\theta\ge 1.$  To do this we may by the Hahn-Banach theorem pick $(x^*_d)_{d\in D}$ with $x_d^*(x+u_d)=\|x+u_d\|$ and $\|x_d^*\|=1.$  We may then pass to a subnet and assume that $(x^*_d)_{d\in D}$ is weak$^*$-convergent to some $x^*$.  Let $x^*_d=x^*+u_d^*.$   Then
\begin{align*} 1&=\lim_{d\in D}(x^*(x)+x^*(u_d)+u_d^*(x)+u_d^*(u_d))\\
&=\lim_{d\in D}(x^*(x)-x^*(u_d)-u_d^*(x)+u_d^*(u_d))\\
&\le \limsup_{d\in D}\|x^*-u_d^*\|\|x-u_d\|\\
&= \theta.\end{align*}
This proves (a).

(b) is a trivial deduction from (a).\end{proof}

\begin{Prop} \label{subspaces} (i) If $X$ is a Banach space with a shrinking 1-unconditional UFDD then $X$ has (au*).\newline (ii) If $Y$ is a separable Banach space with (au*) then any subspace or quotient $X$ of $Y$  also has (au*).\end{Prop}

\begin{proof} (i) is clear, as is (ii) for quotients.  Consider the case when $X$ is a subspace of $Y$.  Suppose $x^*\in X^*$ and $(u_n^*)$ is a weak$^*$ null sequence in $X^*$ such that $\lim_{n\to\infty}\|x^*+u_n^*\|=1$ but $\lim_{n\to\infty}\|x^*-u_n^*\|=1+\delta>1.$  Let $y_n^*\in Y^*$ be extensions to $Y$ with $\|y_n^*\|=\|x^*+u_n^*\|.$ Passing to a subsequence we can suppose $(y_n^*)$ converges weak$^*$ to $y^*.$  Then
$\lim_{n\to\infty}\|2y^*-y_n^*\|=1.$  However $(2y^*-y_n^*)|_X=x^*-u_n^*$ and we have a contradiction.\end{proof}

{\bf Remark.} Note that property $(au)$ does not pass to quotients since every separable Banach space is a quotient of $\ell_1.$

We close this section with a simple Lemma, which will be useful later.
\begin{Lem}\label{weaklynull}
(i) Let $X$ be a separable Banach space with property $(au)$,
 and suppose that $(x_n)_{n=1}^{\infty}$ is a weakly null sequence which is not norm convergent to $0$.  Then, given $\delta>0,$ there is a subsequence $(y_n)_{n=1}^{\infty}$ of $(x_n)_{n=1}^{\infty}$ such that the sequence $(y_n)_{n=1}^{\infty}$  is $(1+\delta)-$unconditional.\newline
 (ii) Let $X$ be a separable Banach space with property $(au^*)$, and suppose that $(x^*_n)_{n=1}^{\infty}$ is a weak$^*$-null sequence in $X^*$ which is not norm convergent to $0$.  Then, given $\delta>0,$ there is a subsequence $(y^*_n)_{n=1}^{\infty}$ of $(x^*_n)_{n=1}^{\infty}$ such that the sequence $(y^*_n)_{n=1}^{\infty}$  is $(1+\delta)-$unconditional.\end{Lem}

\begin{proof} The proofs of these statements are essentially identical so we prove only (i).

 We may suppose, by passing to a subsequence, that $(x_n)_{n=1}^{\infty}$ is basic (see e.g. \cite{AlbiacKalton2006} Theorem 1.5.2). Let $K$ be the basis constant for the sequence $(x_n)_{n=1}^{\infty}$ and assume that $0<c\le \|x_k\|\le C<\infty$ for all $k.$

Choose  $(\delta_n)_{n=1}^{\infty}$ to be a decreasing sequence of positive numbers so that $\prod_{j=1}^{\infty}(1+\delta_j)<1+\delta.$  We will construct a subsequence $(y_n)_{n=1}^{\infty}$ and a sequence $(F_n)_{n=1}^\infty$ of closed finite-codimensional subspaces inductively.

Let $y_1=x_1$ and $F_1=X.$  If $y_1,\ldots,y_{n-1}$ and $F_1,\ldots,F_{n-1}$ have been chosen then we may choose a closed subspace $F_n$ of finite codimension so that if $w\in [y_j]_{j=1}^{n-1}$ and $z\in F_n$ then
$$ \|w-z\|\le (1+\tfrac14\delta_{n})\|w+z\|.$$
Let $Q_j:X\to X/F_j$ denote the quotient map for $1\le j\le n.$  If $y_{n-1}=x_{m_n}$ we may pick $y_n=x_{m_{n+1}}$ with $m_{n+1}>m_n$ so that $$\|Q_jy_n\|\le \frac{2^{j-n-1}c\delta_j}{10K}, \qquad 1\le j\le n.$$

Now suppose $w=\sum_{j=1}^{n-1}a_jy_j$ and $z=\sum_{j=n}^{N}a_jy_j$ where $\|w+z\|=1.$
Then we have
$$\|Q_nz\|= \|\sum_{j=n}^{N}a_jQ_ny_j\| \le 2Kc^{-1} \sum_{j=n}^{\infty} \|Q_ny_j\|\le \delta_n/5.$$

Hence there exists $z'\in F_n$ such that $\|z-z'\|\le \delta_n/4$ and thus
$$ \|w-z\|\le \|w-z'\|+\tfrac14\delta_n\le  (1+\tfrac14\delta_n)\|w+z'\|+\tfrac14\delta_n\le 1+\delta.$$
Thus we have the inequality
\begin{equation}\label{nonzerosum} \|\sum_{j=1}^{n-1}a_jy_j-\sum_{j=n}^Na_jy_j\|\le (1+\delta_n)\|\sum_{j=1}^Na_jy_j\|.\end{equation}

Then we claim that if $\epsilon_j=\pm 1$ with $\epsilon_j=1$ for $j<k$ we have
\begin{equation}\label{product0} \|\sum_{j=1}^n\epsilon_ja_jy_j\| \le \prod_{j=k}^\infty(1+\delta_{j})\|\sum_{j=1}^na_jy_j\|.\end{equation}
This is proved for fixed $n$ by backwards induction on $k.$  Indeed for $k=n$ it follows from \eqref{nonzerosum}.
If it is proved for $k+1$ we simply note that when $\epsilon_k=-1$ but $\epsilon_j=1$ for $j<k,$
$$ \|\sum_{j=1}^n\epsilon_ja_jy_j\|\le (1+\delta_k)\|\sum_{j=1}^ka_jy_j-\sum_{j=k+1}^n\epsilon_ja_jy_j\|\le \prod_{j=k}^{\infty}(1+\delta_j)\|\sum_{j=1}^na_jy_j\|.$$
\end{proof}

\vskip.5truein

\section{Embedding in a space with unconditional basis}\setcounter{equation}{0}

Let $Y$ be a space with an (FDD) $(Q_j)_{j=1}^{\infty}$ and let $X$ be a subspace of $Y.$  Then we will say that $X$ satisfies the {\it density condition} with respect to $(Q_j)_{j=1}^{\infty}$ if there is a dense subset $D$ of $X$ such that if $x\in D$ we have
$$ x=\sum_{j=1}^nQ_jx$$ for some $n=n(x)\in\mathbb N.$  The following Lemma is similar to Lemma 2.1 in \cite{Grivaux2003}.

\begin{Lem}\label{density}  Let $Y$ be a space with an (FDD) $(Q_j)_{j=1}^{\infty}$ and let $X$ be a subspace of $Y.$  Then given $\delta>0$ there exists an automorphism $T:Y\to Y$ so that $\|T-I\|<\delta$ and $X$ satisfies the density condition with respect to the FDD $(TQ_jT^{-1})_{j=1}^{\infty}.$
\end{Lem}

\begin{proof}  We first prove the following claim:

{\it Claim:  Suppose $(Q'_j)_{j=1}^{\infty}$ is any (FDD) of $Y$ and $x\in Y.$  Then given $n\in\mathbb N$ and $\nu>0$ there exists an automorphism $S:Y\to Y$, with $\|S-I\|<\nu$, such that $SQ_j'S^{-1}(Y)=Q_j'(Y)$ for $1\le j\le n$ and for some $m\ge n$ we have
$x\in \sum_{j=1}^mSQ_j'S^{-1}(Y).$}

{\it Proof of the claim:}  Let $K$ be the FDD-constant of $(Q'_j)_{j=1}^{\infty}.$
If $\sum_{j=1}^nQ'_jx=x$ we take $m=n$ and $S=I.$  If not we may choose $m>n$ so that
$$ \|x-\sum_{j=1}^mQ'_jx\|< \nu (2K)^{-1} \|\sum_{j=n+1}^mQ'_jx\|.$$
Pick $y^*\in Y^*$ with $\|y^*\|=1$ and $y^*(\sum_{j=n+1}^mQ'_jx)=\|\sum_{j=n+1}^{m}Q'_jx\|.$  Then let
$$ Sy =y+\|\sum_{j=n+1}^{m}Q'_jx\|^{-1} y^*(\sum_{j=n+1}^mQ'_jy)(x-\sum_{j=1}^mQ'_jx).$$
Then $\|S-I\|<\nu.$  Also $SQ'_j=Q'_j$ if $j=1,2\ldots,n$ so that $SQ'_jS^{-1}(Y)=Q'_j(Y)$.  We also have $S\sum_{j=1}^mQ'_jx=x.$  Hence $S^{-1}x =\sum_{j=1}^mQ'_jx$ and so $x=\sum_{j=1}^mSQ'_jS^{-1}x.$  This concludes the proof of the claim.

We now turn to the Lemma.
Now suppose $\nu_n>0$ are such that $\prod_{j=1}^{\infty}(1+\nu_j)<1+\delta.$  Let $(x_n)_{n=1}^{\infty}$ be a dense sequence in $X.$ We inductively define automorphisms $S_n:Y\to Y$ with $\|S_n-I\|<\nu_n$ and a nondecreasing sequence of integers $(m_n)_{n=0}^{\infty}$ such that if $T_0=I$ and then $T_n=\prod_{j=1}^nS_j$ we have $$T_{n}Q_jT_{n}^{-1}(Y)=T_{n-1}Q_jT_{n-1}^{-1}(Y)$$ for $1\le j\le m_{n-1}$, and $$x_n=\sum_{j=1}^{m_{n}}T_nQ_jT_n^{-1}x_n
.$$
To do this pick $m_0=1$, say and then proceed inductively using the previous claim.  If $m_0,\ldots,m_{n-1}$ and $S_1,\ldots,S_{n-1}$ have been chosen, we pick $S_n$ by the claim so that $\|S_n-I\|<\nu_n,$ $S_nT_{n-1}Q_jT_{n-1}^{-1}S_n^{-1}(Y)=T_{n-1}Q_jT_{n-1}^{-1}(Y)$ for $1\le j\le m_{n-1}$  and for suitable $m_n\ge m_{n-1}$ we have  $$x_n=\sum_{j=1}^{m_n}T_nQ_jT_n^{-1}x_n.$$

The sequence $(T_n)$ converges in operator norm to an operator $T$ where $\|T-I\|\le \prod_{n=1}^{\infty}(1+\nu_n)-1<\delta.$  Clearly $$TQ_jT^{-1}=T_nQ_jT_n^{-1}, \qquad 1\le j\le m_{n}$$ so that
for each $n$,
$$ x_n=\sum_{j=1}^{m_{n}}TQ_jT^{-1}x_n.$$
\end{proof}

\begin{Prop}\label{shrinking} Let $X$ be a separable Banach space containing no copy of $\ell_1$ (respectively a separable reflexive Banach space) which is isometrically embedded in a Banach space $Y$ with a 1-UFDD $(Q_j)_{j=1}^{\infty}.$ Suppose $X$ satisfies the density condition with respect to $(Q_j)_{j=1}^{\infty}.$ Then $X$ can be isometrically embedded into a Banach space $Z$ (respectively a reflexive Banach space) with a shrinking 1-UFDD $(Q'_j)_{j=1}^{\infty}$ with $\rank Q'_j\le \rank Q_j$.

If further $X$ is $\lambda-$complemented in $Y$ then $X$ is $\lambda-$complemented in $Z$. \end{Prop}

\begin{proof}  We assume, without loss of generality, that $Q_j(Y)=Q_j(X)$ for each $j.$ Let $J:X\to Y$ be an isometric embedding. Define on $Y^*$ the norm
$$ |||y^*|||= \sup_{n}\sup_{\epsilon_j=\pm 1}\|\sum_{j=1}^n\epsilon_jJ^*Q_j^*y^*\|.$$
Then $|||\cdot|||$ is weak$^*$-lower semicontinuous and we can define a Banach space $(\tilde Z,\|\cdot\|_Z)$ continuously embedded in $Y$ by
$$ \|z\|_{Z}= \sup\{|y^*(z)|:\ |||y^*|||\le 1\}.$$ By assumption $Q_j(Y)=Q_j(X)\subset \tilde Z$. If we let $Z$ be the closed linear span of $\cup_{j=1}^{\infty}Q_j(Y)$ in $\tilde Z$ then $Z^*$ can be identified with the completion of $(Y^*,|||\cdot|||).$

Clearly $(Q_j)_{j=1}^{\infty}$ is a 1-UFDD for $Z$.
We must check that $(Q_j)_{j=1}^{\infty}$ is shrinking for $Z$.  Indeed if not we can find a blocked sequence
$z_j^*\in \sum_{i=N_{j-1}+1}^{N_j} Q_i^*(Y^*)$ where $N_0=0<N_1<N_2<\cdots$ which is equivalent to the canonical basis $(e_j)_{j=1}^{\infty}$ of $c_0.$  Choose $\epsilon_i=\pm1$ so that
$$ \|J^* \sum_{i=N_{j-1}+1}^{N_j}\epsilon_i Q_i^*z_j^*\|=|||z_j^*|||.$$
If we let $$x^*_j=J^*\left( \sum_{i=N_{j-1}+1}^{N_j}\epsilon_i Q_i^*z_j^*\right)$$ then there is a bounded linear operator $T:c_0\to X^*$ with $Te_j=x^*_j.$
Since $X^*$ contains no copy of $c_0$ this implies that $\|x^*_j\|=|||z_j^*|||$ converges to zero, contrary to assumption.

Also if $x\in X$ then $\|x\|_Y=\|x\|_Z$ so that $X$ is isometrically embedded in $\tilde Z$.  Since a dense subset of $X$ lies in the linear span of $Q_j(Y)$ it follows that $X\subset Z.$  Further since $\|z\|_Z\ge \|z\|_Y$ in general, if there is a projection $P:Y\to X$ with $\|P\|=\lambda$ then $\|P\|_{Z\to X}\le \lambda.$

Finally if $X$ is reflexive we show that $Z$ is reflexive.  To do this it is necessary to show that the UFDD of $Z^*$ given by $(Q_j^*(Z^*))_{j=1}^{\infty}$ is also shrinking.  Suppose not.  Then we can find a blocked sequence
$z_j^*\in \sum_{i=N_{j-1}+1}^{N_j} Q_i^*(Y^*)$ where $N_0=0<N_1<N_2<\cdots$ which is equivalent to the canonical basis $(e_j)_{j=1}^{\infty}$ of $\ell_1.$   Let $\Delta$ denote the Cantor set $\{-1,+1\}^{\mathbb N}$ of all sequences $\epsilon=(\epsilon_i)_{i=1}^{\infty}$ and consider the compact Hausdorff space $\Omega=\Delta\times B_{X}$ where $B_X$ has the weak topology.  Let $f_j\in \C(\Omega)$ be defined by
$$ f_j(\epsilon,x) = \langle x, J^*(\sum_{i=N_{j-1}+1}^{N_j}\epsilon_iQ_i^*z_j^*)\rangle.$$
Then $(f_j)_{j=1}^{\infty}$ is equivalent to the $\ell_1-$basis and so there exists a probability measure $\mu$ on $\Omega$ and a Borel function $\varphi\in L_1(\mu)$ so that
$$ \int_\Omega f_j\varphi\,d\mu\ge 1, \qquad j=1,2,\ldots.$$
We argue that $\lim_{j\to\infty}f_j(\epsilon,x)=0$ for every $(\epsilon,x)\in \Omega$ and this contradicts the Dominated Convergence Theorem.  Indeed
\begin{align*} |f_j(\epsilon,x)|&=|\langle  \sum_{i=N_{j-1}+1}^{N_j}\epsilon_iQ_iJx,z_j^*\rangle|\\
&\le \|\sum_{i=N_{j-1}+1}^{N_j}\epsilon_iQ_iJx\|_Z\|z_j^*\|_{Z^*}\to 0\end{align*}
since $\sum_{i=1}^{\infty}\epsilon_iQ_iJx$ converges.

\end{proof}

The proof of the next Proposition is standard.

\begin{Prop}\label{embedding}  Let $X$ be a separable Banach space.  Then the following conditions on $X$ are equivalent:
\newline
(i) Given $\delta>0$ there exists a Banach space $Y$ with a 1-UFDD and a subspace $X_{\delta}$ of $Y$ with $d(X,X_{\delta})<1+\delta.$
\newline
(ii) Given $\delta>0$ there exists a Banach space $Y$ with a 1-unconditional basis and a subspace $X_{\delta}$ of $Y$ with $d(X,X_{\delta})<1+\delta.$
\newline
(iii) Given $\delta>0$ there exists a Banach space $Y$ containing $X$ (isometrically) and a sequence of finite-rank operators $A_n:X\to Y$ such that
$$ \|\sum_{j=1}^n\epsilon_jA_j\|<1+\delta, \qquad \epsilon_j=\pm 1,\ n=1,2,\ldots$$ and
$$ x=\sum_{j=1}^{\infty}A_jx, \qquad x\in X.$$
\end{Prop}

\begin{proof} (i) $\implies$ (ii): it is essentially contained in \cite{LindenstraussTzafriri1977} Theorem 1.g.5 (p. 51) that every Banach space with a 1-UFDD is $(1+\delta)$-isomorphic to a subspace of a space with a 1-unconditional basis; in \cite{LindenstraussTzafriri1977} the constants are not tracked, but clearly the same argument would prove this more precise statement.

(ii) $\implies$ (iii): this is clear.

(iii) $\implies$ (i): we first note that by blocking the series $\sum A_j$ suitably (i.e. replacing $(A_j)_{j=1}^{\infty}$ by $A'_j=\sum_{N_{j-1}+1}^{N_j}A_i$ for suitable $N_0=0<N_1<\cdots,$ we can assume that for a dense set of $x\in X$ we have $\sum_{j=1}^{\infty}\|A_jx\|<\infty$.  We define $Z$ to be the space of all sequences $(y_j)_{j=1}^{\infty}$ with $y_j\in A_j(X)$ such that $\sum_{j=1}^{\infty}y_j$ converges unconditionally in $Y$, under the norm $$\|(y_j)_{j=1}^{\infty}\|=\sup_{n}\sup_{\epsilon_j=\pm1} \|\sum_{j=1}^n\epsilon_jy_j\|.$$  This space has a 1-UFDD. $X$ can be $(1+\delta)$-embedded into $Z$ via the map $x\to (A_jx)_{j=1}^{\infty}$ (it suffices to note that $Z$ is closed in the larger space of weakly unconditionally Cauchy series with the same norm, and a dense subset of $X$ is mapped into $Z$ by our assumptions).
\end{proof}

Combining  Proposition \ref{shrinking}, Proposition \ref{embedding} and Lemma \ref{density} gives the following result. Part (ii) is contained in Corollary IV.4 of \cite{GodefroyKalton1997} (where the proof is inaccurate); as remarked in \cite{LindenstraussTzafriri1977} p.51 one cannot hope for (ii) to hold with $Y$ having an unconditional basis.

\begin{Prop}\label{cleanup}  Suppose $X$ is a separable Banach space containing no complemented copy of $\ell_1.$  Suppose, given $\delta>0$ there exists a Banach space $Y$ containing $X$ (isometrically) and a sequence of finite-rank operators $A_n:X\to Y$ such that
$$ \|\sum_{j=1}^n\epsilon_jA_j\|<1+\delta, \qquad \epsilon_j=\pm 1,\ n=1,2,\ldots$$ and
$$ x=\sum_{j=1}^{\infty}A_jx, \qquad x\in X.$$ Then
\newline
(i) For any $\delta>0,$ there is a Banach space $Z$ with a shrinking 1-unconditional basis and a subspace $X_{\delta}$ of $Z$ such that $d(X,X_{\delta})<1+\delta.$ \newline
(ii) If $X$ is reflexive then we may take $Z$ reflexive in (i).\newline
(iii) If for every $\delta>0$ we can take $Y=X$ (i.e. $X$ has (UMAP)) then for any $\delta>0,$ there is a Banach space $Z$ with a shrinking 1-UFDD and a $(1+\delta)-$complemented subspace $X_{\delta}$ of $Z$ such that $d(X,X_{\delta})<1+\delta.$
\newline
(iv) If $X$ is reflexive then we may take $Z$ reflexive in (iii).
\end{Prop}

\vskip.5truein

\section{The main result}\setcounter{equation}{0}

\begin{Lem}\label{products}  Let $Y$ be a Banach space and suppose $X$ is a closed subspace of $Y$.  Denote by $Q$ the quotient map $Q:Y\to Y/X.$  Suppose $(B_n)_{n=1}^{\infty}$ is a uniformly bounded sequence of operators on $Y,$ such that
\newline
\begin{equation}\label{10} \lim_{n\to\infty}\|QB_n\|_{X\to Y/X}=0\end{equation} and
\begin{equation}\label{11} \limsup_{n\to\infty}\|B_n\|_{X\to Y}\le 1.\end{equation}
Then given $\delta>0$ there is an infinite subset $\mathbb M$ of $\mathbb N$ such that if $n_1<n_2<\cdots<n_k$ with $n_j\in \mathbb M$ for $1\le j\le k$ then
$$ \|B_{n_1}B_{n_2}\ldots B_{n_k}\|_{X\to Y}< 1+\delta.$$
\end{Lem}

\begin{proof} We suppose $\|B_n\|\le M$ for all $n.$ We assume $\delta<1/2.$ It suffices to prove this for $\mathbb M=\mathbb N$ when $\|QB_n\|<\nu_n/3$ and
$\|B_n\|_{X\to Y}\le 1+\nu_n/3$ where $(\nu_n)_{n=1}^{\infty}$ is the decreasing positive sequence given by $\nu_{n}=(3M+6)^{-n+1}\delta.$  We will prove by induction on $k$ that
\begin{equation}\label{12} \|QB_{n_1}\ldots B_{n_k}\|_{X\to Y/X}<\nu_{n_1}\end{equation} and
\begin{equation}\label{13} \|B_{n_1}\ldots B_{n_k}\|_{X\to Y}<1+\nu_{n_1}.\end{equation}

Under these hypotheses the conclusion is obviously true for $k=1.$  We next assume it is true for $k$ and prove it for products of length $k+1.$  Consider $m<m_1<\cdots<m_k.$  Then if $S=B_{m_1}\ldots B_{m_k}$ we have
$$ \|QS\|_{X\to Y/X}<\nu_{m+1},\qquad \|S\|_{X\to Y}<1+\nu_{m+1}.$$
Now if $x\in X$ with $\|x\|\le 1$ there exists $x'\in X$ so that
$$ \|x'-Sx\|<\nu_{m+1}$$ and then
$$ \|x'\|<1+2\nu_{m+1}.$$
Now
$$ B_m Sx= B_m x'+ B_m(Sx-x')$$ and so we have
\begin{align*} \|QB_mSx\| &< \tfrac13\nu_m(1+2\nu_{m+1})+ M\nu_{m+1}\\ &<\tfrac23\nu_m+M\nu_{m+1}<\nu_m\end{align*} and
\begin{align*} \|B_mSx\| &\le (1+\tfrac13\nu_m)(1+2\nu_{m+1})+ M\nu_{m+1}\\ &<1+\tfrac23\nu_m+(M+2)\nu_{m+1}=1+\nu_m\end{align*} establishing both inductive hypotheses \eqref{12} and \eqref{13}.
\end{proof}

\begin{Thm} \label{main} Let $X$ be a separable  Banach space.  Then the following conditions are equivalent:
\newline (i) $X$ has (au*).\newline
(ii) For any $\delta>0$ there is a Banach space $Y$ with a shrinking 1-unconditional basis and a subspace $X_{\delta}$ of $Y$ such that $d(X,X_{\delta})<1+\delta.$\end{Thm}

\begin{proof}  That (ii) $\implies$ (i) follows from Proposition \ref{subspaces}.  We turn to the proof of (i) $\implies$ (ii).

By Proposition \ref{separabledual} $X^*$ is separable. We start by using the result of Zippin \cite{Zippin1988} that $X$ can be embedded in a space $Y$ with a shrinking basis (we can assume the embedding is isometric).
Let $S_n$ denote the partial sum operators with respect to this basis, and let $Q:Y\to Y/X$ be the quotient map.  We also denote by $J$ the inclusion $J:X\to Y.$

We will prove the following Lemma:

\begin{Lem}\label{inter} Given $\nu>0$ and $n\in\mathbb N$ there exists $T$ in the convex hull of $\{S_{k}:\ k>n\}$ such that
$\|QT\|_{X\to Y/X}<\nu$ and $\|I-2T\|_{X\to Y}<1+\nu.$\end{Lem}

\begin{proof}[Proof of the Lemma]  First we will argue that for every $n\in\mathbb N$ there exists $m>n$ such that
\begin{equation}\label{21} \|J^*(S_n^*y^*+S_m^*y^*-y^*)\| \le \|J^*(S_n^*y^*-S_m^*y^*+y^*)\|+\tfrac12\nu\|y^*\|, \quad y^*\in Y^*.\end{equation}

If \eqref{21} fails we may find a sequence $(y_m^*)_{m>n}$ such that $\|y_m^*\|=1$ and
$$ \|J^*(S_n^*y_m^*+S_m^*y_m^*-y_m^*)\|>\|J^*(S_n^*y_m^*-S_m^*y_m^*+y_m^*)\|+\tfrac12\nu, \qquad m>n.$$  We may pass to a subsequence $\mathbb M$ of $\{n+1,n+2,\ldots\}$ so that $\lim_{m\in\mathbb M}y_m^*=y^*$ weak$^*$ for some $y^*\in Y^*$.  Since $S_n$ is finite rank $\lim_{m\in\mathbb M}\|S_n^*(y^*-y_m^*)\|=0.$
Hence
$$ \liminf_{m\in\mathbb M}(\|J^*S_n^*y^*+J^*(S_m^*y_m^*-y_m^*)\|-\|J^*S_n^*y^*-J^*(S_m^*y_m^*-y_m^*)\|)\ge \tfrac12\nu.$$
Now $(S_m^*y_m^*-y_m^*)_{m=1}^{\infty}$ is weak$^*$-null in $Y^*$ since
for $y\in Y,$
$$ |\langle y, S_m^*y_m^*-y_m^*\rangle|=|\langle S_my-y,y_m^*\rangle|\le \|S_my-y\|.$$
Hence the sequence $(J^*(S_m^*y_m^*-y_m^*))_{m=1}^{\infty}$ is weak$^*$-null  in $X^*$.  Thus we have a contradiction to $(au^*)$ for $X$.
This shows that \eqref{21} holds for some $m=m(n)>n.$

Let us put $R_n=S_{m(n)}-S_n.$  Thus we have
\begin{align*} \|J^*(I-2S_n)^*y^*\| &\le \|J^*(S_n^*y^*+S_m^*y^*-y^*)\|+\|R_n^*y^*\|\\
&\le \|J^* (y^*-S_m^*y^*+S_n^*y^*)\|+\|R_n^*y^*\|+\tfrac12\nu\|y^*\|\\
&\le \|J^*y^*\| + 2\|R_n^*y^*\|+\tfrac12\nu\|y^*\|.\end{align*}
Thus we have
\begin{equation}\label{22}\|J^*(I-2S_n)^*y^*\|\le (1+\tfrac12\nu)\|y^*\| + 2\|R_n^*y^*\|, \qquad y^*\in Y^*.\end{equation}

We next consider two sequences of finite-rank operators.  First we consider the sequence $(QS_nJ)_{n=1}^{\infty}$ in $\mathcal K(X,Y/X)$. Note that if $z^*\in (Y/X)^*$ then $J^*Q^*z^*=0.$  Since $\lim_{n\to\infty}\|S_n^*Q^*z^*-Q^*z^*\|=0$ (as the basis is shrinking) we conclude that $\lim_{n\to\infty}\|J^*S_n^*Q^*z^*\|=0.$  This implies \cite{Kalton1974} that $(QS_nJ)_{n=1}^{\infty}$ is a weakly null sequence in $\mathcal K(X,Y/X).$

Next consider $\tilde R_n:c_0(Y)\to Y$ defined by $\tilde R_n(y_k)_{k=1}^\infty= R_ny_n.$  Then $\tilde R_n^*:Y^*\to\ell_1(Y^*)$ is given by $\tilde R_n^*y^*= (0,\ldots,0,R_n^*y^*,0,\ldots)$ with the non-zero entry in the $n$th slot.
Since the basis of $Y$ is shrinking we have $\lim_{n\to\infty}\|R_n^*y^*\|=0$ for $y^*\in Y^*$ and so also
$\lim_{n\to\infty}\|\tilde R_n^*y^*\|=0.$  This implies that $(\tilde R_n)_{n=1}^{\infty}$ is weakly null in $\mathcal K(c_0(Y),Y)$ again using \cite{Kalton1974}.

Combining these statements with Mazur's theorem for any $n$ we can find $r>n$ and $(\alpha_j)_{j=n+1}^r$ with $\alpha_j\ge 0$ and $\sum_{j=n+1}^r\alpha_j=1$ such that
$$ \|\sum_{j=n+1}^r\alpha_j QS_jJ\|<\nu,\ \quad \|\sum_{j=n+1}^r \alpha_j\tilde R_j\|<\tfrac14\nu.$$
Let $T=\sum_{j=n+1}^r\alpha_jS_j.$  Then $\|QT\|_{X\to Y/X}=\|QTJ\|<\nu.$

Also if $y^*\in Y^*,$ using \eqref{22},
\begin{align*}\|J^*(I-2T)^*y^*\|&\le \sum_{j=n+1}^r\alpha_j\|J^*(I-2S_j)^*y^*\|\\
&\le (1+\tfrac12\nu)\|y^*\|+2\sum_{j=n+1}^r\alpha_j\|R_j^*y^*\| \\
&=(1+\tfrac12\nu)\|y^*\|+2\|\sum_{j=n+1}^r\alpha_j\tilde R_j^*y^*\|.\end{align*}
By the selection of $\alpha_j$ this implies that $\|(I-2T)J\|=\|I-2T\|_{X\to Y}<1+\nu.$
This completes the proof of the Lemma.\end{proof}

We now turn to the proof of the Theorem.  Using  Lemma \ref{inter} and Lemma \ref{products} we can find a sequence of convex combinations
$$ T_j=\sum_{i=N_{j-1}+1}^{N_j}\alpha_iS_i$$ where $N_0=0<N_1<N_2<\cdots$ and $\alpha_i\ge 0$ are such that $
\sum_{i=N_{j-1}+1}^{N_j}\alpha_i=1$ for all $j$ with the property that
$$ \|(I-2T_{n_1})(I-2T_{n_2})\ldots(I-2T_{n_k})\|_{X\to Y}<1+\delta$$ whenever $n_1<n_2<\cdots<n_k.$
Note that the $(T_j)_{j=1}^{\infty}$ are a commuting approximating sequence in $Y$ and that $T_jT_k=T_k$ if $j>k.$

Let $A_j=T_j-T_{j-1}$ where $T_0=0.$  We now repeat a calculation in \cite{CasazzaKalton1990} Theorem 3.8 with a correction to a small misprint.  Note that if $\epsilon_j=\pm1$ we have
$$\epsilon_nT_n\prod_{j=1}^{n-1}(I-T_{n-j}+\epsilon_{n-j+1}\epsilon_{n-j}T_{n-j})=\sum_{j=1}^n\epsilon_jA_j.$$
(Here the index $n-j+1$ replaces $n-j-1.$)
 Since $T_n=\frac12(I-(I-2T_n))$, it follows that
$$ \|\sum_{j=1}^n\epsilon_jA_j\|_{X\to Y}<1+\delta.$$  The result now follows by Proposition \ref{cleanup}.
\end{proof}

\begin{Cor}\label{reflex}  Let $X$ be a reflexive Banach space.  Then $X$ has property $(au)$ if and only if for any $\delta>0$ there is a reflexive Banach space $Y$ with a 1-unconditional basis and a subspace $X_{\delta}$ of $Y$ such that $d(X,X_{\delta})<1+\delta.$\end{Cor}

\begin{proof} This follows from the Theorem and Propositions \ref{*-auau} and \ref{cleanup}.\end{proof}

The next Corollary is due to Johnson and Zheng \cite{JohnsonZhenginprep} by a quite different proof.

\begin{Cor}\label{quotients}  Any quotient of a Banach space $X$ with a shrinking unconditional basis is isomorphic to a subspace of a Banach space with a shrinking unconditional basis.\end{Cor}

\begin{proof} $X$ can be renormed to have $(au^*)$ and so this follows from Proposition \ref{subspaces}.\end{proof}

\vskip.5truein

\section{Skipped unconditional bases}\setcounter{equation}{0}

Let us say that a basic sequence $(e_k)_{k=1}^{N}$ (where $1\le N\le \infty$) in a (finite or infinite-dimensional) Banach space $X$ is {\it skipped $\lambda$-unconditional} if whenever $0=m_0<m_1<\cdots<m_n<\infty$ with $m_j-m_{j-1}\ge 2$ for $1\le j\le n$  and $y_j\in [e_i]_{m_{j-1}+1}^{m_j-1}$ then for any choice of signs $(\epsilon_j)_{j=1}^n,$
$$ \|\sum_{j=1}^n\epsilon_jy_j\|\le \lambda \|\sum_{j=1}^ny_j\|.$$  We shall say that $(e_k)_{k=1}^{\infty}$ is {\it asymptotically skipped 1-unconditional} if for every $\lambda>1$ there exists $n$ so that if $x\in [e_k]_{k=1}^n\setminus\{0\}$  then the basic sequence $\{x,e_{n+1},e_{n+2},\ldots\}$ is skipped $\lambda-$unconditional.

We will define a basis $(f_k)_{k=1}^N$ of a finite-dimensional Banach space to be {\it dual skipped $\lambda$-unconditional} when the dual basis $(f_k^*)_{k=1}^N$ is skipped $\lambda-$unconditional.  We will need the following simple Lemma:

\begin{Lem}\label{elem}  Let $X$ be a Banach space with a  basis $(e_k)_{k=1}^N$ where $1\le N\le\infty.$  Suppose $1\le m_1<m_2<\cdots<m_n<N,$ and that for every $x\in [e_j]_{j=1}^{m_1}$ the basic sequence $\{x,(e_k)_{k=m_1+1}^{N}\}$ is skipped $\lambda-$unconditional.   Suppose $x^*,y^*\in X^*\setminus\{0\}$ are such that $x^*\in [e_k^*]_{k=1}^{m_1},\ y^*(e_j)=0$ for $1\le j\le m_n.$  Then $\{x^*,e^*_{m_2},\ldots,e^*_{m_{n-1}},y^*\}$ is a dual skipped $\lambda$-unconditional basis of its linear span.

In particular if $(e_k)_{k=1}^N$ is skipped $\lambda-$unconditional then the finite sequence $\{x^*,e^*_{m_2},\ldots,e^*_{m_{n-1}},y^*\}$ is a dual skipped $\lambda$-unconditional basis of its linear span.
\end{Lem}

\begin{proof} Define a map $T:X\to \mathbb R^n$ by $$Tx=(x^*(x),e^*_{m_2}(x),\ldots,e^*_{m_{n-1}}(x),y^*(x))$$ and consider the quotient norm $\|\xi\|=\inf\{\|x\|:Tx=\xi\}$ on $\mathbb R^n.$  Then it is easy to check that the canonical basis of $\mathbb R^n$ is skipped $\lambda$-unconditional and its biorthogonal functionals are isometric to
$\{x^*,e^*_{m_2},\ldots,e^*_{m_{n-1}},y^*\}$.\end{proof}

Our next result concerns the unconditionality of the biorthogonal sequence $(e_k^*)_{k=1}^{\infty}$ in $X^*$.  If $\mathbb A$ is a finite subset of $\mathbb N$ we denote by $\text{ubc}(e_j^*)_{j\in\mathbb A}$ the unconditional basis constant of $(e_j^*)_{j\in\mathbb A}.$

\begin{Lem}\label{biorth} Let $X$ be a finite-dimensional Banach space with a skipped $1-$unconditional basis $(e_k)_{k=1}^{2N+1}$. Suppose $\text{ubc}(e_{2j-1}^*)_{j=1}^{N+1}=\mu>1.$   Then
$$ \text{ubc}(e_j^*)_{j=1}^{2N+1}\ge 1 +2(\mu-1).$$
\end{Lem}

\begin{proof}  By assumption there exist real numbers $(\alpha_j)_{j=1}^{N+1}$ and signs $(\epsilon_j)_{j=1}^{N+1}$ so that
$$ \|\sum_{j=1}^{N+1}\alpha_je_{2j-1}^*\|=1$$ and $$ \|\sum_{j=1}^{N+1}\epsilon_j\alpha_{j}e_{2j-1}^*\|=\mu.$$
Let $E=[e_{2j-1}]_{j=1}^{N+1}.$  Then we have
$$ \|\sum_{j=1}^{N+1}\alpha_je_{2j-1}^*|_E\|\le 1$$ and so by the skipped unconditionality condition,
$$ \|\sum_{j=1}^{N+1}\epsilon_j\alpha_{j}e_{2j-1}^*|_E\|\le 1.$$  By the Hahn-Banach theorem there exists $(\beta_j)_{j=1}^N$ such that
$$ \|\sum_{j=1}^{N+1}\epsilon_j\alpha_{j}e_{2j-1}^*+\sum_{j=1}^N\beta_je_{2j}^*\|\le 1.$$
Thus
$$ 2\mu \le 1+\|\sum_{j=1}^{N+1}\epsilon_j\alpha_{j}e_{2j-1}^*-\sum_{j=1}^N\beta_je_{2j}^*\|$$ so that
$$ \text{ubc}(e_j^*)_{j=1}^{2N+1}\ge 2\mu-1.$$
\end{proof}

\begin{Lem}\label{induct} Suppose $N\in\mathbb N$.  Let $X$ be a Banach space of dimension $2^{N}+1$ with a dual skipped $1-$unconditional basis $(f_k)_{k=1}^{2^N+1}$. Suppose $\text{ubc}(f_1,f_{2^N+1})=\mu>1.$   Then
$$ \text{ubc}(f_j)_{j=1}^{2^N+1}\ge 1 +2^N(\mu-1).$$
\end{Lem}

\begin{proof}  This is proved by induction on $N.$  If $N=1$ it is immediate from Lemma \ref{biorth}.  Suppose now that the Lemma is proved for $N-1.$  Then $\{f_1,f_3,\ldots,f_{2^N+1}\}$ is a dual skipped $1-$unconditional basis of its linear span by Lemma \ref{elem}.  By the inductive hypothesis
$$ \text{ubc}(f_{2j-1})_{j=1}^{2^{N-1}+1}\ge 1+ 2^{N-1}(\mu-1).$$
Now applying Lemma \ref{biorth} we have
$$ \text{ubc}(f_{j})_{j=1}^{2^{N}+1}\ge 1+2^{N}(\mu-1).$$\end{proof}

\begin{Prop}\label{skipped2} Let $X$ be a Banach space containing no copy of $\ell_1$ and with a skipped unconditional basis $(e_k)_{k=1}^{\infty}.$  Then:\newline (i) $(e_k)_{k=1}^{\infty}$ is shrinking, and\newline
(ii) If $X$ contains no copy of $c_0$ then either $X$ is reflexive or $X$ is quasi-reflexive of order one.\end{Prop}

\begin{proof} (i) Let $(u_k)_{k=1}^{\infty}$ be any normalized block basic sequence with respect to $(e_k)_{k=1}^{\infty}.$  Then $(u_{2k})_{k=1}^{\infty}$ (respectively $(u_{2k-1})_{k=1}^{\infty}$) is an unconditional basic sequence and hence weakly null; thus $(u_k)_{k=1}^{\infty}$ is weakly null.  Hence $(e_k)_{k=1}^{\infty}$ is shrinking.

(ii) We may assume $\|e_k\|=1$ for all $k.$ Suppose $x^{**}\in X^{**}$ is such that $\lim_{k\to\infty} x^{**}(e_k^*)=0.$  Select a strictly increasing sequence  $(m_k)_{k=0}^{\infty}$ (with $m_0=0$) such that $|x^{**}(e_{m_k}^*)|<2^{-k}$ for $k\ge 1.$  Then the series
$$\sum_{k=1}^{\infty}\left(\sum_{i=m_{k-1}+1}^{m_k-1}x^{**}(e_i^*)e_i\right)$$ is a WUC series and hence convergent in $X$.  On the other hand the series $\sum_{k=1}^{\infty} x^{**}(e^*_{m_k})e_k$ is absolutely convergent and so $x^{**}\in X.$

Now suppose $X$ is non-reflexive and $x_0^{**}\in X^{**}\setminus X.$  Then $\liminf_k |x_0^{**}(e_k^*)|>0.$ For any $x^{**}\in X^{**}$ we may find $\lambda\in\mathbb R$ so that $\liminf_k |(x^{**}-\lambda x_0^{**})(e_k^*)|=0$ and hence $x^{**}-\lambda x_0^{**}\in X.$
Thus $\dim X^{**}/X=1.$\end{proof}

\begin{Prop}\label{skippedau*} Let $X$ be a Banach space containing no copy of $\ell_1$ and with a normalized asymptotically skipped 1-unconditional basis $(e_k)_{k=1}^{\infty}.$ Suppose $X$ fails to have property $(au^*)$.  Then:\newline
(i) No subsequence of $(e_k^*)_{k=1}^{\infty}$ is unconditional, and \newline
(ii) Every spreading model of $(e_k)_{k=1}^{\infty}$ is equivalent to the standard $\ell_1-$basis.
\end{Prop}

\begin{proof} Since $(e_k)_{k=1}^{\infty}$ is shrinking by Proposition \ref{skipped2} we can assume  the existence of $\mu>1$, $r\in\mathbb N$, $\alpha,\beta\in\mathbb R,$ $x^{*}\in [e_k^*]_{k=1}^r$ and a sequence $(y_n^*)_{n>r}$ with $y_n^*(e_j)=0$ for $j<n$ such that $\|x^*\|=\|y_n^*\|=1$ and $\|\alpha x^*-\beta y_n^*\|=1$ for all $n$ but $\|\alpha x^*+\beta y_n^*\|\ge \mu.$ Let $K$ be the basis constant of $(e_k)_{k=1}^{\infty}.$ We first argue that for $n\in\mathbb N$ with $n>80K/(\mu-1)$ there exists $k=k(n)$ so that if $k<m_1<m_2<\cdots<m_n$ then
\begin{equation}\label{ubc} \text{ubc}(e^*_{m_1},\ldots,e^*_{m_n})\ge \frac{(\mu-1)n}{10K^2}.\end{equation}

Assume not.  Then for each $k>r$ we may select $k<m_{k,1}<\ldots<m_{k,n}$ so that
$$ \text{ubc}(e^*_{m_{k,1}},\ldots,e^*_{m_{k,n}})\le \frac{(\mu-1)n}{10K^2}.$$  Hence since the basis constant of $(x^*,e^*_{m_{k,1}},\ldots,e^*_{m_{k,n}},y^*_{m_{k,n}+1})$ is at most $K$ we have that if $\epsilon_1,\epsilon_2,\ldots,\epsilon_{n+2}=\pm1$ and $\xi_1,\ldots,\xi_{n+2}$ are real numbers,
\begin{align*} &\|\epsilon_1\xi_1x^*+\sum_{j=1}^n\epsilon_{j+1}\xi_{j+1}e^*_{m_{k,j}}+\epsilon_{n+2}\xi_{n+2}y_{m_{k,n}+1}^*\|\\ &\le
|\xi_1|+|\xi_{n+2}|+\frac{(\mu-1)n}{10K^2}\|\sum_{j=1}^n\xi_{j+1}e^*_{m_{k,j}}\|\\
&\le \left(4K +\frac{(\mu-1)n}{5}\right)\|\xi_1x^*+ \sum_{j=1}^n\xi_{j+1}e^*_{m_{k,j}}+\xi_{n+2}y_{m_{k,n}+1}^*\|\\
&\le \frac{(\mu-1)n}{4}\|\xi_1x^*+ \sum_{j=1}^n\xi_{j+1}e^*_{m_{k,j}}+\xi_{n+2}y_{m_{k,n}+1}^*\|.\end{align*}
Thus
$$ \text{ubc}(x^*,e^*_{m_{k,1}},\ldots,e^*_{m_{k,n}},y^*_{m_{k,n}+1})\le \frac{\mu-1}{4}n.$$
Note the basis $(x^*,e^*_{m_{k,1}},\ldots,e^*_{m_{k,n}},y^*_{m_{k,n}+1})$ is dual $\lambda_k-$skipped where $\lim_k\lambda_k=1.$

Let us define a norm on $\mathbb R^{n+2}$ by
$$ \|(\xi_1,\ldots,\xi_{n+2})\|=\lim_{\mathcal U}\|\xi_1x^*+\sum_{j=1}^n\xi_{j+1}e^*_{m_{k,j}}+\xi_{n+2}y^*_{m_{k,n}+1}\|$$ where $\mathcal U$ is some non-principal ultrafilter.  The canonical basis  $(f_1,\ldots,f_{n+2})$ is then dual skipped 1-unconditional and $$\text{ubc}(f_1,\ldots,f_{n+2})\le \frac14(\mu-1)n.$$  Also $\text{ubc}(f_1,f_{n+2})\ge \mu.$  Hence by Lemma \ref{induct} (and utilizing Lemma \ref{elem} since $n+1$ need not be a power of $2$)
$$ \text{ubc}(f_1,\ldots,f_{n+2}) \ge \frac12(\mu-1)(n+1).$$
This gives a contradiction and \eqref{ubc} is established.

(i) is now immediate.

For (ii) observe that any spreading model of $(e_k)_{k=1}^{\infty}$ is 1-unconditional.  For any $n$ there exists $k(n)$ so that \eqref{ubc} holds.  Suppose $k(n)<m_1<\cdots<m_n.$  Then there exist $(\alpha_1,\ldots,\alpha_n)\in\mathbb R^n$ and $(\epsilon_1,\ldots,\epsilon_n)\in\{-1,1\}^n$  so that
$$ \|\sum_{j=1}^n\alpha_je^*_{m_j}\|=1,\qquad \|\sum_{j=1}^n\epsilon_j\alpha_je^*_{m_j}\|\ge \frac{\mu-1}{10K^2}n.$$
Since $\|e^*_k\|\le 2K$ we thus have
$$ \sum_{j=1}^n|\alpha_j|\ge \frac{\mu-1}{20K^3}n$$ and so
for a suitable choice of signs $\eta_j$ we have
$$ \|\sum_{j=1}^n\eta_je_{m_j}\| \ge \frac{\mu-1}{20K^3}n.$$
Thus in any spreading model with basis $(f_j)_{j=1}^\infty$ we have $\|f_1+\cdots+f_n\|\ge cn$ for suitable $c>0.$  This implies that $(f_j)_{j=1}^{\infty}$ is equivalent to the canonical basis of $\ell_1$ (since it is a 1-unconditional spreading model).
\end{proof}

\vskip.5truein

\section{The Weak Alternating Banach-Saks Property}\setcounter{equation}{0}

We recall that a Banach space $X$ is said to have the {\it Alternating Banach-Saks (ABS) property} if every bounded sequence $(x_n)_{n=1}^{\infty}$ in $X$ has a subsequence $(y_n)_{n=1}^{\infty}$ such that
\begin{equation}\label{ABS} \lim_{n\to\infty}\sup_{r_1<r_2<\cdots<r_n} \|\frac1n\sum_{j=1}^n(-1)^jy_{r_j}\|=0.\end{equation}
This is equivalent to the requirement that some spreading model of $(x_n)_{n=1}^{\infty}$ is not equivalent to the $\ell_1-$basis (see \cite{Beauzamy1979}).

We shall say that $X$ has the {\it Weak Alternating Banach-Saks (WABS) property} if every bounded sequence $(x_n)_{n=1}^{\infty}$ in $X$ has a convex block sequence $(y_n)_{n=1}^{\infty}$ such that \eqref{ABS} holds.  Here $(y_n)_{n=1}^{\infty}$ is a convex block sequence if
$$ y_n=\sum_{j=p_{n-1}+1}^{p_n}\lambda_jx_j$$ where $p_0=0<p_1<p_2<\cdots$, $\lambda_j\ge 0,$ and $\sum_{p_{n-1}+1}^{p_n}\lambda_j=1$ for every $n.$
Note that if $(y_n)_{n=1}^{\infty}$ satisfies \eqref{ABS} then so does every further sequence of convex blocks.

Let us recall at this point that a Banach space $X$ has Pe\l czy\'nski's property (u) if for every weakly Cauchy sequence
$\{x_n\}_{n=1}^{\infty}$ there is a weakly null sequence $(z_n)_{n=1}^{\infty}$ so that if $u_n=x_n-z_n$ then the series $\sum_{n=1}^{\infty}(u_n-u_{n-1})$ (where $u_0=0$) is weakly unconditionally Cauchy (WUC). Any Banach space with an unconditional basis has property (u) (\cite{Pelczynski1958}, \cite{LindenstraussTzafriri1979} p.31) Let us note the following, which shows the connection with the (WABS) property:

\begin{Prop}\label{u} Let $X$ be a separable Banach space.  Then $X$ contains no copy of $\ell_1$ and has property (u) if and only if every bounded sequence $(x_n)_{n=1}^{\infty}$ has a convex block sequence $(y_n)_{n=1}^{\infty}$ such that
\begin{equation}\label{u2} \sup_n \sup_{r_1<r_2<\cdots<r_n}\|\sum_{j=1}^n(-1)^jy_{r_j}\|<\infty.\end{equation}\end{Prop}

\begin{proof}  If $X$ contains no copy of $\ell_1$, we can assume  $(x_n)_{n=1}^{\infty}$ is weakly Cauchy \cite{Rosenthal1974}.  If $X$ has property (u) we can write $x_n=u_n+z_n$ where $(z_n)$ is weakly null and $\sum (u_n-u_{n-1})$ is a WUC series.  We may then pass to convex blocks $(\hat x_n)_{n=1}^{\infty}$ so that the corresponding convex blocks  $(\hat z_n)_{n=1}^{\infty}$ and $(\hat u_n)_{n=1}^{\infty}$ satisfy $\|\hat z_n\|<2^{-n}.$  Then $(\hat x_n)_{n=1}^{\infty}$ satisfies our requirements.

Conversely it is clear $X$ cannot contain $\ell_1$.  If $(x_n)_{n=1}^{\infty}$ is weakly Cauchy we may pass to convex blocks $(y_n)_{n=1}^{\infty}$ verifying \eqref{u2}.  But then $\sum(y_n-y_{n-1})$ is a WUC series and $x_n-y_n$ is weakly null.
\end{proof}

In \cite{HaydonOdellRosenthal1991} Haydon, Odell and Rosenthal introduced the class of Baire-1/2 functions: if $\Omega$ is a compact metric space then a bounded function $f$ on $\Omega$ is Baire-1/2 if for every $\epsilon>0$ there exist bounded lower-semi-continuous functions $\varphi,\psi$ such that $|f(s)-(\varphi(s)-\psi(s))|<\epsilon$ for $s\in \Omega.$

Suppose $X$ is a separable Banach space and $x^{**}\in X^{**}\setminus X$. We can generate a sequence $\chi_n=\chi_n(x^{**})\in X^{(2n)}$ by $\chi_1=x^{**}$ and then $\chi_n=j_{n-1}^{**}x^{**}$ where $j_{n-1}$ is the canonical embedding $X\subset X^{**}\subset\cdots\subset X^{2(n-1)}.$ The sequence $(\chi_n)_{n=1}^{\infty}$ is considered in the transfinite dual $X^{\omega}$ defined as the completion of $\cup_{n\ge 1}X^{(2n)}.$

The following theorem follows easily from \cite{HaydonOdellRosenthal1991} and \cite{Farmaki1993}:

\begin{Thm}\label{WABS}  If $X$ is a separable Banach space then the following are equivalent:
\newline
(i) $X$ has the (WABS) property.\newline
(ii) Every $x^{**}\in X^{**}$ is Baire-1/2 as a function on $B_{X^*}$ with the weak$^*$-topology.\newline
(iii) There is no $x^{**}\in X^{**}\setminus X$ so that $(\chi_n(x^{**}))_{n=1}^{\infty}$ is equivalent to the unit vector basis of $\ell_1.$
\end{Thm}

\begin{proof} (i) $\Rightarrow$ (ii).  Since $X$ contains no copy of $\ell_1$, every $x^{**}\in X^{**}\setminus X$ is the weak$^*$-limit of a sequence $(x_n)_{n=1}^{\infty}$ \cite{OdellRosenthal1975}.  We pass to a sequence of convex blocks $(y_n)_{n=1}^{\infty}$ so that \eqref{ABS} holds.  Now apply Theorem B of \cite{HaydonOdellRosenthal1991} to deduce that $x^{**}$ is Baire-1/2.

(ii) $\iff$ (iii).  This is Theorem 11 of  Farmaki \cite{Farmaki1993} (since (iii) also implies that $X$ contains no copy of $\ell_1$ by Proposition 6 of \cite{Farmaki1993}).

(ii) $\Rightarrow$ (i).  Let $(x_n)_{n=1}^{\infty}$ be a bounded sequence in $X$.  If $(x_n)_{n=1}^{\infty}$ has a weakly convergent subsequence then Mazur's theorem quickly yields a sequence of convex blocks satisfying \eqref{ABS}.  By Rosenthal's theorem \cite{Rosenthal1974} we may therefore pass to the case when $(x_n)_{n=1}^{\infty}$ is weakly Cauchy and converging weak$^*$ to some $x^{**}\in X^{**}\setminus X.$   By Theorem 3.7 and Lemma 3.8 of \cite{HaydonOdellRosenthal1991} there is a bounded sequence $(f_n)_{n=1}^{\infty}$ in $\C(B_{X^*})$ converging pointwise to $x^{**}$ so that $(f_n)_{n=1}^{\infty}$ satisfies \eqref{ABS}.  By Mazur's theorem, we may find a sequence of convex blocks $(y_n)_{n=1}^{\infty}$ of $(x_n)_{n=1}^{\infty}$ and a sequence of convex blocks  $(g_n)_{n=1}^{\infty}$ of $(f_n)_{n=1}^{\infty}$ such that $\|y_n-g_n\|<2^{-n}$ (considering $X$ as a subspace of $\C(B_{X^*})).$  Then $(y_n)_{n=1}^{\infty}$ satisfies \eqref{ABS}.
\end{proof}

We next give a very similar argument to Lemma \ref{weaklynull} for the case when $(x_n)_{n=1}^{\infty}$ converges weak$^*$ to some $x^{**}\in X^{**}\setminus X.$

\begin{Lem}\label{auskipped}  Let $X$ be a separable Banach space with property $(au)$,
 and suppose that $(x_n)_{n=1}^{\infty}$ is a weakly Cauchy sequence in $X$ converging weak$^*$ to some $x^{**}\in X^{**}\setminus X.$  Then there is a subsequence $(y_n)_{n=1}^{\infty}$ of $(x_n)_{n=1}^{\infty}$ such that the sequence $(y_n-y_{n-1})_{n=1}^{\infty}$ (where $y_0=0$) is an asymptotically skipped 1-unconditional basic sequence.\end{Lem}

\begin{proof}  We may suppose, by passing to a subsequence, that $(x_n)_{n=1}^{\infty}$ is basic (see e.g. \cite{AlbiacKalton2006} Theorem 1.5.6), and that if $x^*\in X^*$ is such that $x^{**}(x^*)=1$ then $|x^*(x_n)-1|<2^{-n}$.  This implies the existence of $y^*\in X^*$ with $y^*(x_n)=1$ for all $n$ and so $(x_n-x_{n-1})_{n=1}^{\infty}$ (with $x_0=0$) is  also a basic sequence (see \cite{Singer1970} pp. 308-311); note this remark applies to all subsequences of $(x_n)_{n=1}^{\infty}.$  Let $K$ be the basis constant for the sequence $(x_n)_{n=1}^{\infty}$ and assume that $0<c\le \|x_k\|\le C<\infty$ for all $k.$

Let $(\delta_n)_{n=1}^{\infty}$ be a decreasing sequence of positive numbers with the property that $\sum_{n=1}^{\infty}\delta_n<\infty.$  We will construct a subsequence $(y_n)_{n=1}^{\infty}$ and a sequence $(F_n)_{n=1}^\infty$ of closed finite-codimensional subspaces inductively.

Let $y_1=x_1$ and $F_1=X.$  If $y_1,\ldots,y_{n-1}$ and $F_1,\ldots,F_{n-1}$ have been chosen then we may choose a closed subspace $F_n$ of finite codimension so that if $w\in [y_j]_{j=1}^{n-1}$ and $z\in F_n$ then
$$ \|w-z\|\le (1+\tfrac14\delta_{n})\|w+z\|.$$
Let $Q_j:X\to X/F_j$ denote the quotient map for $1\le j\le n.$  If $y_{n-1}=x_{m_n}$ we may pick $y_n=x_{m_{n+1}}$ with $m_{n+1}>m_n$ so that $$\|Q_jy_n-Q_j^{**}x^{**}\|\le \frac{2^{j-n-1}c\delta_j}{10K}, \qquad 1\le j\le n.$$

Now suppose $w=\sum_{j=1}^{n-1}a_jy_j$ and $z=\sum_{j=n}^{N}a_jy_j$ where $\|w+z\|=1$ and $\sum_{j=n}^Na_j=0.$
Then we have
$$\|Q_nz\|= \|\sum_{j=n}^{N}a_j(Q_ny_j-Q_n^{**}x^{**})\| \le 2Kc^{-1} \sum_{j=n}^{\infty} \|Q_ny_j-Q_n^{**}x^{**}\|\le \delta_n/5.$$

Hence there exists $z'\in F_n$ such that $\|z-z'\|\le \delta_n/4$ and thus
$$ \|w-z\|\le \|w-z'\|+\tfrac14\delta_n\le  (1+\tfrac14\delta_n)\|w+z'\|+\tfrac14\delta_n\le 1+\delta.$$
Thus we have the inequality
\begin{equation}\label{zerosum} \|\sum_{j=1}^{n-1}a_jy_j-\sum_{j=n}^Na_jy_j\|\le (1+\delta_n)\|\sum_{j=1}^Na_jy_j\|,\qquad \text{if } \sum_{j=n}^Na_j=0.\end{equation}

Now let $z_n=y_n-y_{n-1}$ and suppose $v_j=\sum_{m_{j-1}+1}^{m_j-1}a_jz_j$ for $1\le j\le n$ where $m_0=0<m_1<\cdots<m_n$ with $m_j-m_{j-1}\ge 2$ for $j\ge 2.$
Then we claim that if $\epsilon_j=\pm 1$ we have
\begin{equation}\label{product} \|\sum_{j=1}^n\epsilon_jv_j\| \le \prod_{j=1}^n(1+\delta_{m_j})\|\sum_{j=1}^nv_j\|.\end{equation}

This is proved by induction on $n\ge 2.$ For $n=2$ it follows from \eqref{zerosum}.  Assume it is proved for $n-1.$  Then
\begin{align*} \|\sum_{j=1}^n\epsilon_jv_j\|&\le (1+\delta_{m_1})\|v_1+v_2+\sum_{j=3}^n\epsilon_2\epsilon_jv_j\|\\
&\le \prod_{j=1}^n(1+\delta_{m_j})\|\sum_{j=1}^nv_j\|.\end{align*}

Hence $(y_j-y_{j-1})_{j=1}^{\infty}$ is asymptotically skipped 1-unconditional.\end{proof}

\begin{Thm}\label{WABS2}
Let $X$ be a separable Banach space.  Then the following are equivalent:
\newline (i) $X$ has properties $(au)$ and (WABS),\newline  (ii) For any $\delta>0$ there is a Banach space $Y$ with a shrinking 1-unconditional basis and a subspace $X_{\delta}$ of $Y$ such that $d(X,X_{\delta})<1+\delta.$\end{Thm}

\begin{proof}  Of course by Theorem \ref{main} (ii) is equivalent to the fact that $X$ has $(au^*).$

 (ii) $\Rightarrow$ (i).  We observe that (ii) implies $X$ has property (u) and hence (WABS).  Property $(au)$ follows trivially from (ii).

(i) $\Rightarrow$ (ii). Clearly $X$ contains no copy of $\ell_1.$  Suppose $x^{**}\in X^{**}\setminus X$; by the Odell-Rosenthal theorem \cite{OdellRosenthal1975} and property (WABS) there is a sequence $(x_n)_{n=1}^{\infty}$ converging weak$^*$ to $x^{**}$ with the property that
$$ \lim_{n\to\infty}\sup_{r_1<r_2<\cdots<r_n}\Big\|\frac1n\left(\sum_{k=1}^n(-1)^kx_{r_k}\right)\Big\|=0.$$
According to Lemma \ref{auskipped}, by passing to a subsequence we can assume that $(x_n-x_{n-1})_{n=1}^{\infty}$ is asymptotically skipped 1-unconditional.  But then no spreading model of $(x_n-x_{n-1})_{n=1}^{\infty}$ (with $x_0=0$) is equivalent to the $\ell_1-$basis.  Thus, by Proposition \ref{skippedau*} we have that the space $E=[x_n-x_{n-1}]_{n=1}^{\infty}$ has property $(au^*).$  In particular by Theorem \ref{main} $E$ has property (u). Since $x^{**}$ is in the weak$^*$-closure of $E$ we conclude that $X$ has property (u).

We next show that $X$ has property $(au^*)$.  Suppose not.  Then there exists $x^*\in X^*$ and a weak$^*$-null sequence $(x^*_n)_{n=1}^{\infty}$ such that $\|x^*+x_n^*\|\le 1$ and $\|x^*-x_n^*\|> 1+\delta$ for some $\delta>0.$  Pick $x_n\in X$ so that $\|x_n\|=1$ but $x^*(x_n)-x_n^*(x_n)>1+\delta.$  If $(x_n)_{n=1}^{\infty}$ is weakly convergent to some $x$ then we obtain a contradiction since
$$ \lim_{n\to\infty}x^*(2x-x_n)+x_n^*(2x-x_n)=\lim_{n\to\infty}x^*(x_n)-x_n^*(x_n)>1+\delta$$ but
$$ \lim_{n\to\infty}\|2x-x_n\|=1.$$
Thus we can assume, passing to a subsequence, that  $(x_n)_{n=1}^{\infty}$ is a basic sequence which converges weak$^*$ to some $x^{**}\in X^{**}\setminus X.$   Since $X$ has property (u) there is sequence $(y_n)_{n=1}^{\infty}$ in $X$ so that $(y_n)_{n=1}^{\infty}$ also converges weak$^*$ to $x^{**}$ and is equivalent to the summing basis of $c_0.$  Let $G=[y_n]_{n=1}^{\infty}.$  By Sobczyk's theorem (see \cite{Sobczyk1941} or e.g. \cite{AlbiacKalton2006} Theorem 2.5.8) there is a projection $P:X\to G.$  Then $(P^{**}x_n)_{n=1}^{\infty}$ converges weak$^*$ to $x^{**}$ and so if $Q=I-P$ the sequence $(Qx_n)_{n=1}^{\infty}$ is weakly null.

Now, by Lemma \ref{weaklynull}, passing to a further subsequence of $(x_n)_{n=1}^{\infty}$ we can suppose that either (a) $\|Qx_n\|<2^{-n}$ or (b) $(Qx_n)_{n=1}^{\infty}$ is an unconditional basis for its closed linear span $Z.$ We may also suppose that $z_n=x_n-x_{n-1}$ (where $x_0=0$) defines an asymptotically skipped 1-unconditional basis of $Z.$
In case (a) the space $E=[x_n]_{n=1}^{\infty}$ is isomorphic to a subspace of $c_0.$  In case (b) $E$ is isomorphic to a subspace of $Z\oplus G.$  In either case $E$ embeds (isomorphically, not isometrically) into a space with a shrinking unconditional basis.  In particular the biorthogonal sequence $(z_n^*)_{n=1}^{\infty}$ in $Z^*$ (which is weak$^*$-null) has  a subsequence which is an unconditional basic sequence (again by Lemma \ref{weaklynull}).
By Proposition \ref{skippedau*} this means that $Z$ has property $(au^*)$.  Now
$ \|(x^*+x_n^*)|_Z\|\le 1$ and so $\limsup_{n\to\infty}\|(x^*-x_n^*)|_Z\|\le 1.$  However
$(x^*-x_n^*)(x_n)>1+\delta$ and we have a contradiction.\end{proof}

{\bf Remark.} We do not know whether it is possible to replace the (WABS)-condition in (i) by the assumption that $X$ contains no copy of $\ell_1$ (or even that $X^*$ is separable).  This problem reduces to the question of whether one can find a space $Y$ with an asymptotically skipped 1-unconditional basis, which contains no copy of $\ell_1$ but does not have property $(au^*).$  If one further imposes the condition that $Y$ contains no copy of $c_0$ then $Y$ would be quasi-reflexive of order one by Proposition \ref{skipped2}.  It is certainly possible to find such quasi-reflexive spaces which fail the (WABS) property; this is  the requirement that the transfinite dual $Y^{\omega}\approx Y\oplus \ell_1$.  Examples have been given by Bellenot \cite{Bellenot1982} and by Haydon, Odell and Rosenthal \cite{HaydonOdellRosenthal1991}.  However it seems difficult to impose the extra condition that $Y$ has an asymptotically skipped 1-unconditional basis and therefore leads us to speculate that Theorem \ref{WABS2} can be improved.

Note that the James space \cite{James1951a} (or see \cite{AlbiacKalton2006} p.62) is quasi-reflexive and does have (WABS).  It therefore fails $(au^*)$ (it does not even have property (u)).  By Theorem \ref{WABS2} the James space  cannot have $(au)$ under any equivalent norming.  However it does have the (UTP) of Johnson and Zheng \cite{JohnsonZheng2008}.

{\bf Remark.}  The (WABS)-condition also appears implicitly in \cite{Kalton2007a} where Theorem 4.5 could be rephrased as saying that a separable Banach space with the (WABS) property and the $\mathcal Q-$property is reflexive; this implies that if $X$ is a space with the (WABS) property such that $X$ coarsely embeds into a reflexive space or  $B_X$ uniformly embeds into a reflexive space then $X$ is reflexive.  There is a clear link with the problems considered here.  For example if $X$ is a separable Banach space with an unconditional basis containing no copy of $c_0$ then  $B_X$ uniformly embeds in a reflexive space (Theorem 3.8 of \cite{Kalton2007a}).

\end{document}